\documentclass
[12pt]{amsart}
\usepackage{amsfonts,latexsym,amssymb,
}
\newtheorem{theorem}{Theorem}%[section]

\newtheorem{proposition}[theorem]{Proposition} 
 
\newtheorem{corollary}[theorem]{Corollary}

\theoremstyle{definition}
\newtheorem{definition}[theorem]{Definition}

\newtheorem{problem}[theorem]{Problem} 
\newtheorem{problems}[theorem]{Problems}

\theoremstyle{remark}

\newcommand{\brfrt}{\hspace{0 pt}}
\DeclareMathOperator{\cf}{cf}

\newcommand{\m}{\mathfrak}

\begin{document}
 
\title
% [Orderings of Ultrafilters]
{Some more Problems about Orderings of Ultrafilters}

\author{Paolo Lipparini} 
\address{Dipartimento  Matematic\ae\\Viale Ricerc\ae\ Scientific\ae\\II Universit\`a di Roma (Tor Vergata)\\I-00133 ROME ITALY}
\urladdr{http://www.mat.uniroma2.it/\textasciitilde lipparin}

\thanks{The results and problems presented in this note have been announced at
the Logical Meeting in honor of
Annalisa
Marcja, held in Florence, May 6    and 7, 2010. The author wishes to express his most sincere thanks the organizers for the invitation. He also wishes to express his warmest gratitude to Annalisa Marcja and Piero Mangani for their great support and encouragement in studying Logic and Model Theory.
The author has received support from MPI and GNSAGA. He wishes to express  his gratitude to Xavier Caicedo and Salvador Garcia-Ferreira for stimulating discussions and correspondence} 



\begin{abstract} 
 We discuss the connection between various orders on the class of
all the ultrafilters and certain compactness
properties of abstract logics and of topological spaces. We present a
model theoretical characterization of Comfort order.
We introduce a new order motivated by considerations in abstract model theory.
For each of the above orders, we show that 
if $E$ 
is a $( \lambda , \lambda )$-regular ultrafilter,
and $D$ 
is not $( \lambda , \lambda )$-regular, then
 $E \not \leq  D$.
Many  problems are stated.
\end{abstract} 
 
\maketitle  
 
We refer to \cite{C,CK,CN,Eb,gfpams,GF,GS,mru,Ma} for unexplained notions.

Many orderings on the class of all ultrafilters have been introduced. 
All of these orderings can be viewed from many different points of view, 
and lead to equivalent formulations of some notions, either in purely
ultrafilter theoretical terms, in topological terms, or in model theoretical terms.
We discuss some of these connections, introduce still another order motivated by
abstract model theoretical considerations, and state some further problems.

Throughout, let $D$ be an ultrafilter over some set $I$,
and $E$ be an ultrafilter over some set $J$. 

We first  recall the definition of the classical Rudin-Keisler order.

\begin{definition} \label{rk} 
For   $D$ and $E$ ultrafilters, the  \emph{Rudin-Keisler} (pre-)order
is defined as follows.

$E \leq _{RK} D$
 if and only if 
there is a function $f: I \to J$ such that,  
for every $Y \subseteq J$, it happens that
$Y \in E$ if and only if $f ^{-1}(Y) \in D$. 
\end{definition}   

Notice that, in the above situation, the ultrafilter structure of $E$
is completely determined by $f$ and by the ultrafilter structure of $D$.
 If $E \leq _{RK} D$, we sometimes will say that $E$ is a \emph{quotient} of $D$.   
The Rudin-Keisler order can be given several equivalent reformulations.

\begin{theorem} \label{equivrk}
For every pair of ultrafilters $D$ and $E$, the following are equivalent. 
\begin{enumerate}
\item
$E \leq _{RK} D$.
\item
for every model 
$\m A$, we have that
$\prod_E \m A$
is elementarily embeddable in 
$\prod_D \m A$.   
\item
Every $D$-pseudocompact topological space is $E$-pseudocompact.
\item
Every $D$-pseudocompact Tychonoff topological space is $E$-\brfrt pseudocompact. 
 \end{enumerate}  
 \end{theorem}

\begin{proof}
(1) $\Leftrightarrow $   (2) is \cite[Exercise 4.3.41]{CK}.

(1) $\Rightarrow $  (3) $\Rightarrow $  (4) are trivial.

(4) $\Rightarrow $  (1) is immediate from \cite[Lemma 1.4]{GF}. 
\end{proof}

The notion of $D$-compactness
has played a very important role in the study
of compactness properties of topological spaces, particularly in connection
with products. The paper \cite{GS} is a milestone in the field.
The next definition is due to W. Comfort, and appears in 
\cite{gf3,gfpams}.

\begin{definition} \label{comfort}
The \emph{Comfort (pre-)order} is defined as follows.

$E \leq _{C} D $ if and only if every $D$-compact topological space is 
$E$-compact. 
 \end{definition}   

It can be shown that the Comfort order turns out to be the same if 
we restrict ourselves to Tychonoff spaces.
Moreover, Garcia-Ferreira \cite{gf3}
shows that   the Rudin-Keisler order is strictly finer
than the Comfort order.

We now give a model theoretical characterization of Comfort order. 

\begin{definition} \label{kclos}   
If $D$ is an ultrafilter, let us say that a class $K$ of models of the same type is
\emph{$D$-closed} if and only if $K$ is closed under isomorphism, $K$ is closed under elementary substructures, and any ultraproduct by $D$ of members of $K$ still
belongs to $K$.
\end{definition} 

\begin{proposition} \label{modcom}
Suppose that $D$ is an ultrafilter over $I$, and $E$
is an ultrafilter over $J$. Then  the following are equivalent. 

\begin{enumerate}
\item
$E \leq _{C} D$.
\item
Every $D$-closed class of models is $E$-closed.
\item 
For every model $\m A$, the smallest
$D$-closed class containing $\m A$ is $E$-closed. 
\item 
For every model $\m A$ with $|A| = \sup \{ |I|, |J| \} $, the smallest
$D$-closed class containing $\m A$ is $E$-closed. 
\end{enumerate}  

 \end{proposition}

 \begin{proof}
Without loss of generality, we can suppose that $E$ and $D$ are 
(not necessarily uniform) over
the same cardinal $\alpha$.

Garcia-Ferreira \cite[Theorem 2.3]{gf3} and \cite[Theorem 3.3]{gfpams} 
proved that $E \leq _{C} D$ if and only if 
$E \in \beta _D ( \alpha )$, where  
$\beta _D ( \alpha )$ is the $D$-compactification of 
the discrete space $\alpha$, that is, the smallest $D$-compact
subspace of $\beta ( \alpha )$ containing $\alpha$, where
$\beta ( \alpha )$ denotes  the Stone-\v Cech compactification
of $\alpha$.

(1) $\Rightarrow $  (2)
Suppose that $E \in \beta _D ( \alpha )$.
Since $\beta _D ( \alpha )$ is the  smallest $D$-compact
subspace of $\beta ( \alpha )$ containing $\alpha$, it follows that 
every element of $\beta _D ( \alpha )$ can be iteratively constructed 
starting by $\alpha$, and taking $D$-limits
 of ultrafilters already known to be in $\beta _D ( \alpha )$ (see 
\cite{gfpams} for details). It is well-known, already from \cite{Fr}, that the $D$-limit
of certain ultrafilters $(E_i) _{i \in I} $ corresponds to 
a quotient of the sum $ \sum _D E_i$ (see \cite{mru} for the definition).
Model theoretically, if  $ D'=\sum _D E_i$, then, for
every model $\m A$, $\prod _{D'} \m A $ is isomorphic to
$\prod _D \prod _{E_i} \m A$. This implies that
every class of models which is both $D$-closed and $E_i$-closed for every 
$i \in I$ is also $D'$-closed.   

Moreover, if $D'' \leq _{R} D' $, then, for every 
model $\m A$, $\prod _{D''} \m A $ is 
 elementarily embeddable in 
$\prod_{D'} \m A$, by
Theorem \ref{equivrk} (1) $\Rightarrow $  (2).
Thus, every class of models which is $D'$-closed 
 is also $D''$-closed.   

By iterating the above arguments, and by the above description
of $\beta _D ( \alpha )$,
we get that if $E \in \beta _D ( \alpha )$,
then every 
 $D$-closed class of models is $E$-closed.

(2) $\Rightarrow $  (3) $\Rightarrow $  (4) are trivial.

(4) $\Rightarrow $  (1)
By the above arguments, 
for every model $\m A$, the smallest
$D$-closed class containing $\m A$ is the class of all 
isomorphic copies of the models of the form
$ \prod _{E'} \m A$, for $E' \in \beta _D ( \alpha )$.

Let $\m A$ be the complete model of cardinality $\alpha$,
and let $K$ be 
the smallest
$D$-closed class containing $\m A$.
By assumption, $K$ is $E$-closed;
in particular, $\prod _E \m A \in K$.
By the above remark,
$\prod _E \m A \in K$ is isomorphic to
$ \prod _{E'} \m A$, for some $E' \in \beta _D ( \alpha )$.
Now,  since  $|A|= \alpha $, both $E$
and $E'$ are over $\alpha$,  and $\m A$
is a complete structure, then the ultrafilter structures of
$E$
and $E'$ can be recovered from the structures, respectively, of
$\prod _E \m A$ and of
$ \prod _{E'} \m A$.
Since these latter models are isomorphic,
$E$
and $E'$ are isomorphic,
thus
$E \in \beta _D ( \alpha )$, since
$E' \in \beta _D ( \alpha )$.

\end{proof}

\begin{problem} \label{comfnorm}
What does the Comfort order become when we restrict 
ourselves to special classes of topological spaces?

In more detail, if $T$ is a class of topological spaces, let us define
the \emph{Comfort (pre-)order relative to} $T$ as follows.

$E \leq _{T, C} D$ if and only if every 
$D$-compact topological space belonging to $T$
is also $E$-compact.

Does $E \leq _{ T, C} D$ coincide with $E \leq _{C} D$, when $T$
is the class of Hausdorff normal topological spaces? 

Does $E \leq _{T, C} D$ coincide with $E \leq _{C} D$, when $T$
is the class of topological groups? 
 \end{problem}  

In the particular case of ultrafilters over $ \omega$, the order
$\leq _{T, C}$ has been introduced by  Garcia-Ferreira \cite{gfgrp}.
He also asked whether $E \leq _{T, C} D$ coincides with $E \leq _{C} D$,
 when $T$
is the class of topological groups,
and gave a partial affirmative answer.

We now introduce an analogue of the Comfort order, an analogue
which refers to compactness
properties of abstract logics.

By a \emph{logic} we mean an extension of first-order logic satisfying certain
regularity properties (see \cite{Eb} for more details). Examples of logics in the present sense are logics allowing infinitary conjunctions and disjunctions (infinitary logics), or logics obtained by adding new quantifiers (e. g., cardinality logics). 
As far as the present note is concerned, the exact closure properties a logic is required to satisfy are those listed at the beginning of  \cite[Section 2]{C}.
 
Makowsky and Shelah \cite{MS} defined the notion of an
ultrafilter {\it related} to a logic, and found many applications of this notion
(see \cite{Ma} for a survey). Essentially, an ultrafilter is related to a logic if and only if a version of \L o\'s Theorem holds for that ultrafilter and that logic.
Later \cite{Cprepr,C} gave an improved definition, and extended Makowsky and Shelah's results to a more general setting. We shall usually say that a logic $\mathcal L$ is \emph{$D$-compact}, in place of saying  that $D$ is related to $\mathcal L$.
 
\begin{definition} \label{mkc}
We define as follows the Caicedo-Makowsky-Shelah
\mbox{(pre-)order}.

For ultrafilters $D$ and $E$,
we write $E \leq _{CMS} D$
to mean that every $D$-compact logic is $E$-compact. 
\end{definition}   

The next proposition, asserting  that
$\leq _{C}$ is finer than $E \leq _{CMS} D$,
is an immediate corollary of results from \cite{C}.

\begin{proposition} \label{cmsc}
For every pair of ultrafilters $D$ and $E$, if
$E \leq _{C} D$, then $E \leq _{CMS} D$.
 \end{proposition}  

\begin{proof}
The proposition follows immediately from \cite[Lemma 2.3]{C},
which asserts that $D$-compactness of some logic $\mathcal L$ 
is equivalent to $D$-compactness of certain topological spaces
$E_ \sigma ( {\mathcal L})$ (whose definition depends only on 
$\mathcal L$, and not on $D$).   
\end{proof}

\begin{problems} \label{mainprb}
Does
$E \leq _{C} D$ coincide with $E \leq _{CMS} D$?

Does
$E \leq _{T,C} D$ coincide with $E \leq _{CMS} D$,
for $T$ the class of topological groups?
for $T$ the class of Hausdorff normal topological spaces?

More generally, study the (pre-)order 
$E \leq _{CMS} D$, and  characterize it in 
topological and set theoretical terms.
 \end{problems} 

Notice that, by \cite[Lemma 2.3]{C},
$E \leq _{T,C} D$ coincides with $E \leq _{CMS} D$,
when $T$ is the class of the topological spaces of the form 
$E_ \sigma ( {\mathcal L})$, as defined in \cite{C}. 

We now show that $( \lambda , \lambda )$-regularity constitute a ``dividing line''
for each of the above orderings.

\begin{proposition} \label{partial}
Suppose that $\lambda$ is an infinite cardinal, $D$ 
is not $( \lambda , \lambda )$-regular, and $E$ 
is $( \lambda , \lambda )$-regular. Then the following holds.

\begin{enumerate}
\item 
 $E \not \leq _{ C} D$.
\item
 More generally, $E \not \leq _{T, C} D$,
 where $T$
is the class of Hausdorff normal topological spaces.
\item
 $E \not \leq _{T, C} D$,
 where $T$
is the class of Tychonoff topological groups.
\item
$E \not\leq _{CMS} D$.
 \end{enumerate}  
 \end{proposition}

  \begin{proof}
(1) follows from either (2), (3) or (4).

(2) \cite[Proposition 1]{topproc} and \cite[Corollary 2]{nuotop} constructed a 
Hausdorff normal topological space $X$ such that,
for every ultrafilter $F$, $X$ is $F$-compact if and only if 
$F$ 
is not $( \lambda , \lambda )$-regular.
Thus, $X$ is $D$-compact, but not $E$-compact, hence
 $E   \leq _{T, C} D$ fails.

(3) is similar, using \cite[Proposition 3]{nuotop}.

(4) First suppose that $\lambda$ is regular. Let $ \omega _ \alpha $
be a cardinal such that $ \cf \omega_ \alpha = \lambda $, 
and $ \omega _ \beta ^{|I|} < \omega _ \alpha  $, for every $\beta < \alpha $.
Then classical methods (see, e. g., \cite{arch}) imply that 
$\mathcal L _{ \omega, \omega } (Q _ \alpha )$   is $D$-compact. This logic is not
$E$-compact, since no elementary extension of 
$ \prod _E \langle \lambda , \leq \rangle $ 
can be $\mathcal L _{ \omega, \omega } (Q _ \alpha )$-equivalent to 
 $ \prod _E \langle \lambda , \leq \rangle $.

The case when $\lambda$ is singular is treated in a similar way, by using a logic generated by two cardinality quantifiers, since an ultrafilter 
is $( \lambda , \lambda   )$-regular if and only if it is either 
 $( \cf \lambda , \cf \lambda   )$-regular or $( \lambda^+ , \lambda^+   )$-regular 
\cite{mru}.
 \end{proof} 

Proposition \ref{partial} strongly suggests the hypothesis that the study of
compactness properties both of logics and
of (products of) topological spaces  
actually deals with properties of
 the Comfort and related orders, and that problems
about (transfer of) compactness are best stated as
problems about these orders. Indeed, results stated in terms
of $D$-compactness are more powerful than results
stated in terms of $[ \lambda , \lambda ]$-compactness.
In fact, older results by the author are immediate consequences of 
 Proposition \ref{partial}.

\begin{corollary} \label{transf}
For every infinite cardinals $\lambda$ and $\mu$, the following are equivalent.
\begin{enumerate} 
\item
There exists a $( \lambda , \lambda  )$-regular not $(\mu, \mu)$-regular
ultrafilter.  
\item 
There exists a productively $[ \lambda , \lambda  ]$-compact not
productively  $[\mu, \mu]$-compact topological space.
\item 
There exists a productively $[ \lambda , \lambda  ]$-compact not
productively  $[\mu, \mu]$-compact Tychonoff topological group.
\item 
There exists a productively $[ \lambda , \lambda  ]$-compact not
productively  $[\mu, \mu]$-compact Hausdorff normal topological space.
\item 
There exists a $[ \lambda , \lambda  ]$-compact not
 $[\mu, \mu]$-compact logic.
\end{enumerate}  
 \end{corollary}  

Thus, a more detailed study of the orders 
 $E  \leq _{T, C} D$
and
$E \leq _{CMS} D$
will probably shed more light to  problems connected with compactness of logics and topological spaces.

It will be probably useful also to consider the possibility of dealing
 with more than two ultrafilters at a time.

\begin{problem} \label{multi} 
Study the following relations (which are not pre-orders).

For families $(E_k) _{k \in K}  $
and $(D_h) _{h \in H} $ 
of ultrafilters, let
 $ (E_k) _{k \in K}  \leq _{ C} (D_h) _{h \in H} $
mean that
every topological space which is 
$D_h$-compact, for every $ h \in H$,
is $E_k$-compact, for some  $k \in K$.

The relation
 $ (E_k) _{k \in K}  \leq _{T, C} (D_h) _{h \in H} $,
for $T$ a class of topological spaces is defined similarly. 

Similarly, let  $ (E_k) _{k \in K}  \leq _{ CMS} (D_h) _{h \in H} $ 
mean that
every logic which is 
$D_h$-compact, for every $ h \in H$,
is $E_k$-compact, for some  $k \in K$.
\end{problem} 

\begin{proposition} \label{rel} 
Let $\leq$ be any one of the following relations:
$ \leq _{ C}$, $ \leq _{ CMS}$, or 
$ \leq _{T, C}$, where $T$ is a class of topological 
spaces closed under taking Frechet disjoint unions in the sense of 
\cite[Definition 7]{nuotop}.

Then  $ (E_k) _{k \in K}  \leq _{ } (D_h) _{h \in H} $
if and only if there is some $k \in K$ such that  
$ E_k  \leq _{ } (D_h) _{h \in H} $.  
\end{proposition} 

 \begin{proof} 
The if-condition is trivial.

For the converse, suppose by contradiction that,
for every $k \in K$, we have
$ E_k  \not \leq _{ T,C} (D_h) _{h \in H} $.
Thus, for every  $k \in K$,
there is a topological space $X_ \kappa \in T$
which is not $E_k$-compact, but  
which is $D_h$-compact, for each $h \in H$.   
Then the Frechet disjoint union of the  $X_ \kappa $'s
witnesses the failure of $ (E_k) _{k \in K}  \leq _{ T, C} (D_h) _{h \in H} $,
by \cite[Proposition 8]{nuotop}.

The argument for $ \leq _{ CMS}$ is similar, by taking
a union of logics.
\end{proof}

\end{document}